\documentclass[letterpaper,12pt]{amsart}
\title{Polytopes from Subgraph Statistics}

\author{Alexander Engstr\"om}
\address{Department of Mathematics \\
University of California, Berkeley, CA 94720}
\email{alex@math.berkeley.edu}
\author{Patrik Nor\'en}
\address{KTH -- The Royal Institute of Technology, Stockholm, Sweden}
\email{pnore@math.kth.se}

\date{\today}

\usepackage{latexsym,array,delarray,amsthm,amssymb,epsfig}

\theoremstyle{plain}
\newtheorem{theorem}{Theorem}[section]
\newtheorem{lemma}[theorem]{Lemma}
\newtheorem{proposition}[theorem]{Proposition}
\newtheorem{corollary}[theorem]{Corollary}
\newtheorem{conjecture}[theorem]{Conjecture}

\theoremstyle{definition}
\newtheorem{definition}[theorem]{Definition}

\newtheorem{example}[theorem]{Example}

\theoremstyle{remark}

\begin{document}
\maketitle

\begin{abstract}
Polytopes from subgraph statistics are important in applications and conjectures and theorems in extremal graph theory can be stated as properties of them. We have studied them with a view towards applications by inscribing large explicit polytopes and semi-algebraic sets when the facet descriptions are intractable. The semi-algebraic sets called curvy zonotopes are introduced and studied using graph limits. From both volume calculations and algebraic descriptions we find several interesting conjectures.
\end{abstract}

\section{Introduction}\label{sec:Introduction}

In this paper we study polytopes from subgraph statistics. The vertices of these polytopes are given by the relative proportions of subgraphs of different types.
We got interested in studying the polytopes from subgraph statistics after several questions were raised about them by Rinaldo, Fienberg and Zhou \cite{rinaldoFienbergZhou}. They investigated maximum likelihood estimation for exponential random graph models  and realized that its behavior is closely linked to the geometry of these polytopes. The subgraphs counted are usually determined by the applications of the model, and in the social sciences small graphs as stars and triangles are common \cite{social}, but we make no restrictions of that type.

Our results and methods are from graph theory and discrete geometry, but we have made an effort to address directions that are important for applications. For example to understand when the polytopes have the expected dimensions and how to approximate them with similar polytopes or semi-algebraic sets with easy explicit descriptions when their facet structures are not attainable.

On the track to these descriptions we have found several, for us, unexpected results and conjectures about both enumerative and geometric combinatorics.

For the experts who wants to skip ahead, we want to clarify a notational difference between mathematical communities: Our focus is on finite graphs and finite dimensional polytopes close to applications, even if we use limits and infinite objects as technical tools in some of the later proofs. Therefore the $t$-function counts ordinary honest subgraphs and not graph homomorphisms, and these notions are crucially different in the finite setting before the limit.

\subsection{A short overview of the paper}

In Section~\ref{sec:Basic} we define the polytope from subgraphs statistics and its lattice version. We explain basic properties of their facet structures and how different polytopes and Ehrhart polynomials are related to each other. A complete facet description would solve many open problems in extremal graph theory, so to get any understanding of these polytopes we inscribe more well-studied polytopes and semi-algebraic sets in them. 

The spine of our polytopes is defined in Section~\ref{sec:Spine}. The convex hull of a finite number of points on the spine is a cyclic polytope inscribed in the polytope of subgraph statistics, and the convex hull of the spine is a semi-algebraic set. In Section~\ref{sec:Volumes} we calculate the volumes of the inscribed sets and find both explicit formulas and integrals connected to the Selberg integral formula. The dual of the convex hull of the spine and the polytope provides a method to find certificates that polynomials are non-negative.

The convex hull of the spine never fills up all of the polytope, and can even be of lower dimensions. In Section~\ref{sec:Zonotope} we introduce the curvy zonotopes. They are semi-algebraic sets with explicit descriptions that are of the top dimension. Alternatively the spine could have been defined as the expected values of graphs from the Erd\H{o}s-Renyi graph model with different model parameters. We get the curvy zonotopes as expected values of particular exchangeable graph models.

In Section~\ref{sec:Limit} we show that the curvy zonotopes are not only of the right dimension, but that they can be chosen to get the volume arbitrary close to the polytope that they are inscribed in. The proofs relies on the theory of graph limits and Szemer\'edi regularity. Finally in Section~\ref{sec:Conjectures} we study the limit case of counting complete subgraphs and conjecture that the limit objects essentially are cyclic polytopes with infinite many vertices.

\section{Basic properties of polytopes of subgraph statistics}\label{sec:Basic}

In this section we give proper definitions of the objects of our study, and provide some first results to describe them. We follow standard notation in graph theory, as in for example Diestel \cite{diestel}.

The number of subgraphs of $G$ isomorphic to $F$ is $t^L(F,G)$. Later the letter L indicates that we work with lattice polytopes. The \emph{$F$-subgraph density in $G$} is defined as 
\[t(F,G)=\frac{t^L(F,G)}{t^L(F,K_{|G|})},\]
except when $F$ has more vertices than $G$, and then it is $0$.

\begin{definition}
Let $\mathbf{F}=(F_1,\ldots,F_d)$ be a vector of graphs and $G$ a graph. The vector valued subgraph statistics are
\[t(\mathbf{F},G)=(t(F_1,G),\ldots,t(F_d,G))\]
and
\[t^L(\mathbf{F},G)=(t^L(F_1,G),\ldots,t^L(F_d,G)).\]
\end{definition}

\begin{definition}
Let $\mathbf{F}$ be a vector of graphs and $n$ a positive integer. The \emph{polytope from subgraph statistics} $P_{\mathbf{F};n}$ and its lattice version $P^L_{\mathbf{F};n}$ are defined as
\[P_{\mathbf{F};n}=\mathrm{conv}~\{t(\mathbf{F},G)\mid  \textrm{$G$ is a graph on $n$ vertices}  \}\]
and
\[P^L_{\mathbf{F};n}=\mathrm{conv}~\{t^L(\mathbf{F},G)\mid \textrm{$G$ is a graph on $n$ vertices}  \}.\]
\end{definition}

\begin{example}\label{example:polytopes}
The polytope $P_{ (K_3, C_4, K_4 \setminus e ); 6}$  is drawn in Figure~\ref{fig:polytope}, and in Figure~\ref{fig:skeleton} is a combinatorial representation of its vertices and edges.
\end{example}

If larger examples of polytopes from subgraph statistics looks anything like in Figures~\ref{fig:polytope} and \ref{fig:skeleton}, then it will be very difficult to give an explicit facet description. And indeed many hard theorems and conjectures in extremal graph theory can be rephrased as questions about these polytopes, making a complete facet description probably impossible in general. In Figure~\ref{fig:skeleton} we tabulated the vertices by the actual subgraph counts and not the proportions $t(F,G)$. This defines the lattice polytope $P^L_{\mathbf{F};n}$, a rescaling of $P_{\mathbf{F};n}$. 
\begin{figure}[htbp]
  \begin{center}
    \includegraphics[width=0.75\textwidth]{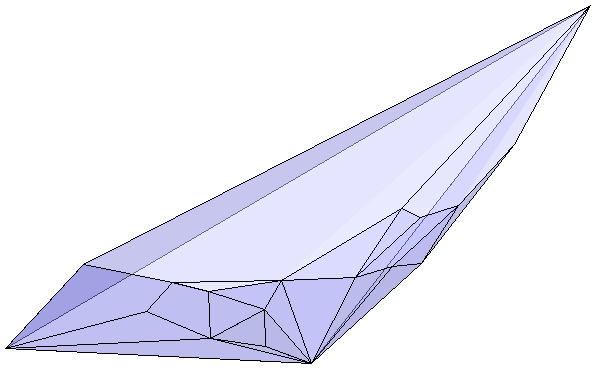}
    \caption{The polytope $P_{ (K_3, C_4, K_4 \setminus e ); 6}.$}
    \label{fig:polytope}
  \end{center}
\end{figure}
\begin{figure}[htbp]
  \begin{center}
    \includegraphics[width=0.75\textwidth]{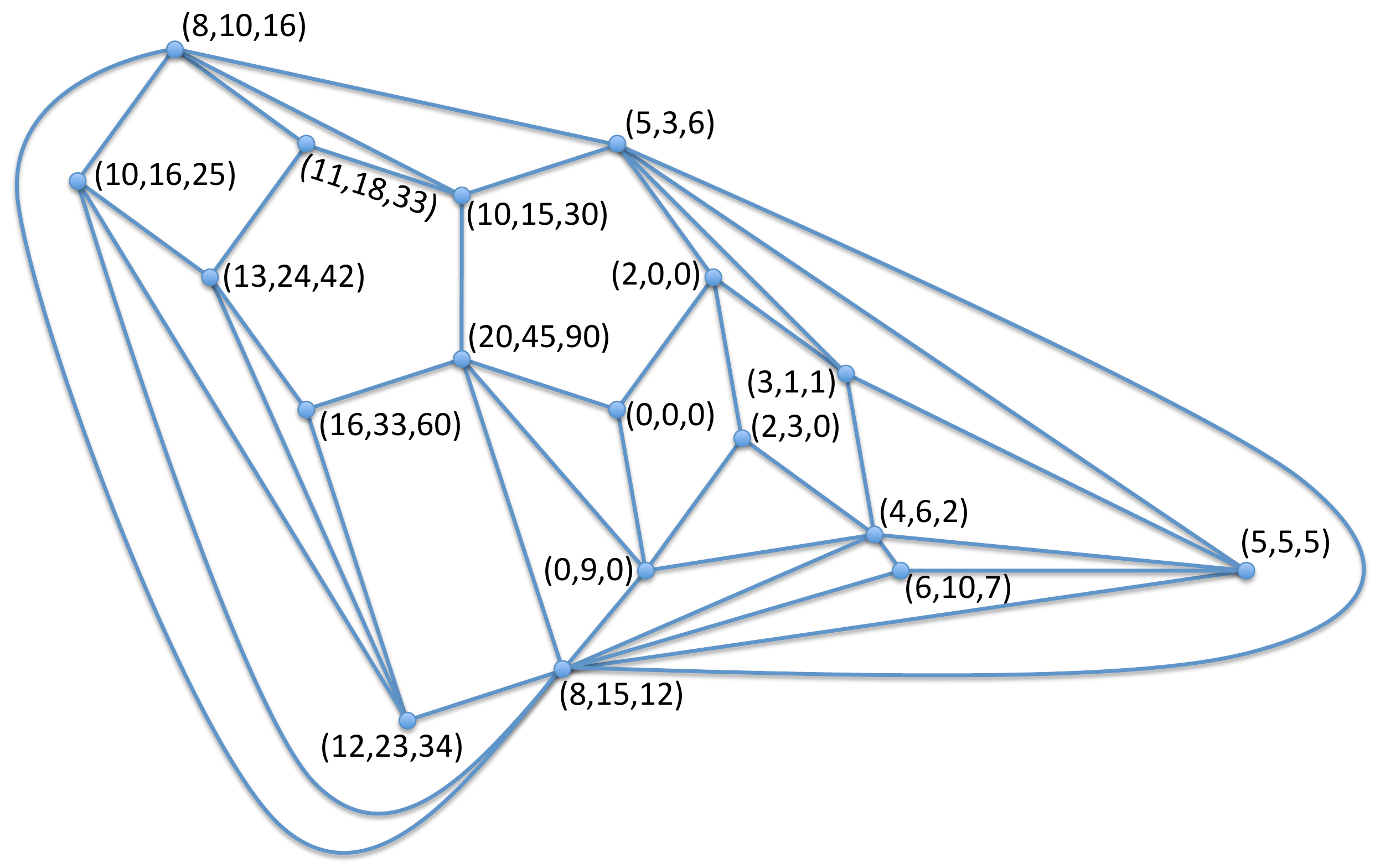}
    \caption{A combinatorial representation of the vertices and edges of $P_{ (K_3, C_4, K_4 \setminus e ); 6},$ indexed by the actual subgraph counts.}
    \label{fig:skeleton}
  \end{center}
\end{figure}
Several graphs could have the same subgraph statistics, and even if $t(\mathbf{F},G_1)$ and $t(\mathbf{F},G_2)$ are different vertices on the same facet, it is not necessary that $G_1$ and $G_2$ are related in any sense, for example as subgraphs. This is illustrated in Figure~\ref{fig:skeletonGraphs}.
\begin{figure}[htbp]
  \begin{center}
    \includegraphics[width=0.75\textwidth]{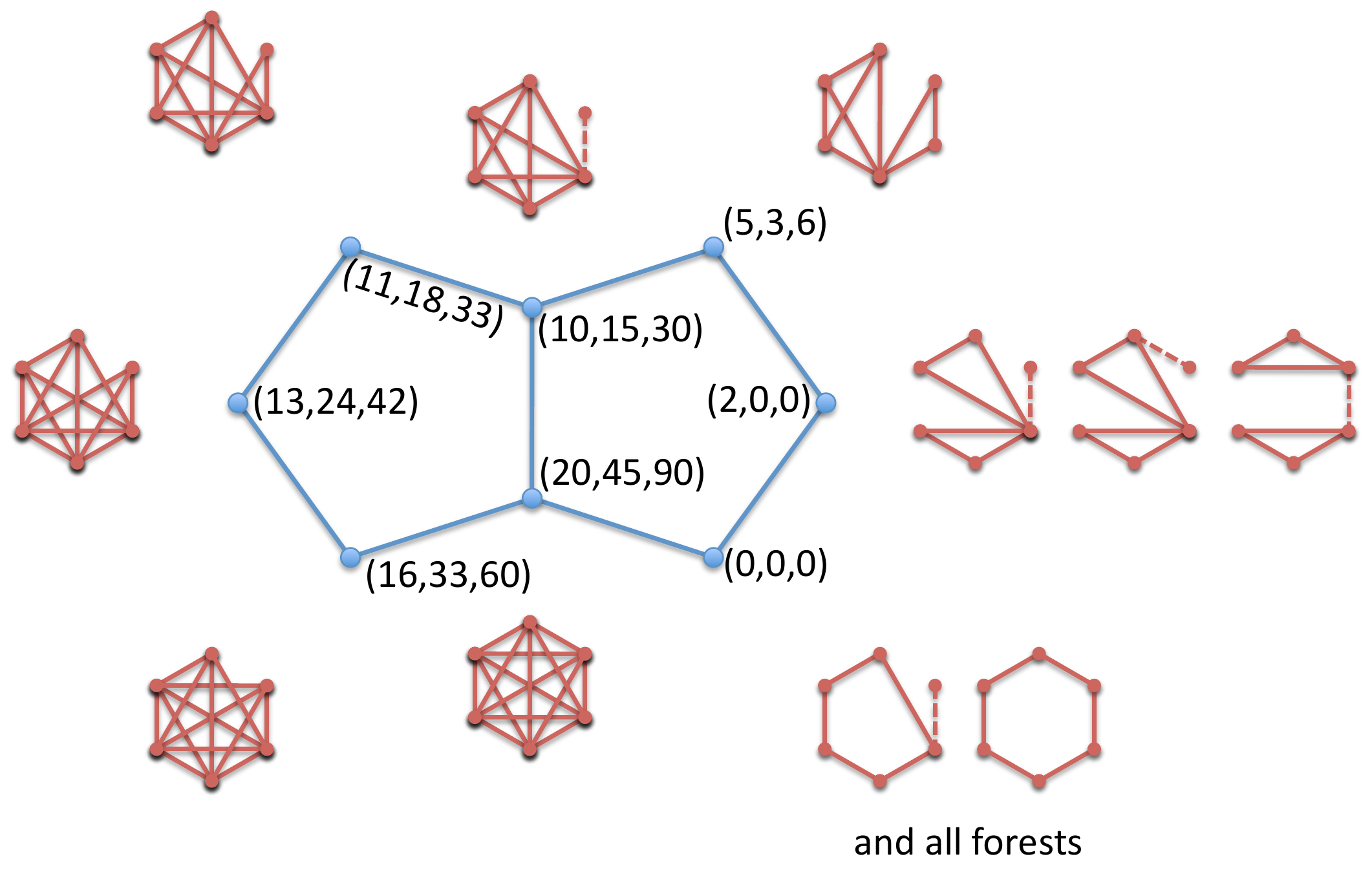}
    \caption{The graphs underlying the statistics of a piece of the polytope in Figures~\ref{fig:polytope} and \ref{fig:skeleton}. Dotted graph edges could be included or not. Recall that the subgraphs counted are $(K_3,C_4,K_4 \setminus e )$.}
    \label{fig:skeletonGraphs}
  \end{center}
\end{figure}

Before embarking on more general results about the polytopes from subgraph statistics, we point out some easy facts.
\begin{proposition}
Let $\mathbf{F}=(F_1,\ldots,F_d)$ be a vector of graphs with edges of order at least $n$, none of them a subgraph of another. Then the inequalities $x_i\ge0$ are facet defining for $P_{\mathbf{F};n}$.
\end{proposition}
\begin{proof}
The number of subgraphs is non-negative and the inequality $x_i\ge0$ always defines a face, the question is if it is a facet. Let $G_j$ be the graph given by the disjoint union of $F_j$ and enough isolated vertices to get $n$ vertices in total. The vector $t(\mathbf{F},G_j)$ is non-zero exactly in the component $j$. On the hyperplane defined by $x_i=0$ is a dimension $d-1$ simplex defined by the points $\mathbf{0}$
and $t(\mathbf{F},G_j)$ for $i\neq j$. 
\end{proof}

\begin{lemma}\label{lemma:LtonL}
If $F$ and $G$ are graphs then $t^L(F,G) = \frac{|F|!}{\mathrm{Aut}(F)} {|G|\choose |F|}t(F,G)$.
\end{lemma}
\begin{proof}
If $t(F,G)=t^L(F,G)=0$ it is true. Otherwise use that $t^L(F,K_{|G|})= \frac{|F|!}{\mathrm{Aut}(F)} {|G|\choose |F|}.$
\end{proof}

\begin{lemma}\label{lemma:folklore}
If $F$ and $G$ are graphs and $n$ an integer then
\[{|G|-|F|\choose n-|F|}t^L(F,G) =\sum_{U\in{V(G)\choose n}}t^L(F,G[U]).\]
\end{lemma}
\begin{proof}
If not $|F|\le n \le|G|$ then both sides are zero. Otherwise, consider a particular copy of $F$ in $G$. On the right hand side it is counted for every subset $U$ of size $n$ containing the vertex set of that copy of $F$. The set $U$ contains $n-|F|$ elements to be chosen among the $|G|-|F|$ vertices outside that particular copy of $F$.
\end{proof}

\begin{proposition}\label{prop:inclusion}
Let  $\mathbf{F}$ be a vector of graphs of order at most $n$.
If $n \leq n' \leq n''$ then $P_{\mathbf{F};n'} \supseteq P_{\mathbf{F};n''}$.
\end{proposition}
\begin{proof}
Let $G$ be a graph on $n''$ vertices and let $F$ be any graph in $\mathbf{F}$. First combine Lemma~\ref{lemma:LtonL} and Lemma~\ref{lemma:folklore} to get
\[t(F,G)= \frac{{n'  \choose |F|}   }{{|G|  \choose |F|}{|G|-|F|\choose n'-|F|}}  \sum_{U\in{V(G)\choose n'}}t(F,G[U]).\]
By expanding and simplifying, one gets that $\frac{{n'  \choose |F|}   }{{|G|  \choose |F|}{|G|-|F|\choose n'-|F|}} = {|G|\choose n'} ^{-1}.$
Collected into vectors this is
\[ t(\mathbf{F},G)= {|G|\choose n'} ^{-1} \sum_{U\in{V(G)\choose n'}}  t(\mathbf{F},G[U])\]
and all vertices of $P_{\mathbf{F};n''}$ are in simplices spanned by points in $P_{\mathbf{F};n'}$. 
\end{proof}

For any number $k,$ a polytope $P$ can be inflated to $kP= \{kp | p\in P \}$. For a lattice polytope $P$ and positive integers $k$, the number of lattice points in $kP$ is the \emph{Ehrhart polynomial} $E_P(k).$ We refer to chapter 12 of Miller and Sturmfels \cite{millerSturmfels} for a proof of this, and a description of the connections to algebraic geometry. This is a translation of Proposition~\ref{prop:inclusion} into the lattice polytope setting.
\begin{proposition}
Let  $\mathbf{F}=(F_1,\ldots,F_d)$ be a vector of graphs of order $n$. Then \[E_{P^L_{\mathbf{F};n''}}\left({ n' \choose n} k \right)  \le E_{P^L_{\mathbf{F};n'}}\left({ n'' \choose n } k \right)\] if $n'' \geq n' \geq n$.
\end{proposition}
\begin{proof}
Define a linear rescaling map $L_\mathbf{F}:\mathbb{R}^d\rightarrow \mathbb{R}^d$ by $L_\mathbf{F}(\mathbf{e}_i) = \mathrm{Aut}(F_i)^{-1} \mathbf{e}_i.$
By Lemma~\ref{lemma:LtonL},
\[ L_\mathbf{F}\left(n!{n''\choose n}P_{\mathbf{F};n''}\right)=P^L_{\mathbf{F};n''} \textrm{ and }L_\mathbf{F}\left(n!{n'\choose n}P_{\mathbf{F};n'}\right)=P^L_{\mathbf{F};n'}.\]
According to Proposition~\ref{prop:inclusion} there is an inclusion of polytopes $P_{\mathbf{F};n''} \subseteq P_{\mathbf{F};n'}$ since $n'' \geq n'$. We rescale the inclusion to 
\[{ n' \choose n}\left(n!{n''\choose n}P_{\mathbf{F};n''}\right)\subseteq { n'' \choose n }\left(n!{n'\choose n}P_{\mathbf{F};n'}\right),\]
apply the linear rescaling map and move out some scalar factors, 
\[{ n' \choose n}L_\mathbf{F}\left(n!{n''\choose n}P_{\mathbf{F};n''}\right)\subseteq { n'' \choose n }L_\mathbf{F}\left(n!{n'\choose n}P_{\mathbf{F};n'}\right),\]
or equivalently, ${ n' \choose n}P^L_{\mathbf{F};n''} \subseteq { n'' \choose n}P^L_{\mathbf{F};n'}$. Counting lattice points gives the desired result.
\end{proof}
In the proposition it is required that all graphs in $\mathbf{F}$ are of the same order, and this can partially be generalized by adding isolated vertices to get graphs of the same order.
\begin{example}
If $\mathbf{F}$ is the graph vector of the path on three vertices and the triangle, then
\[ E_{P_{\mathbf{F};3}^L} \left( { 4 \choose 3 } k \right) =  E_{P_{\mathbf{F};4}^L} \left( { 3 \choose 3 } k \right) = 8k^2+6k+1\]
and
\[E_{P_{\mathbf{F};3}^L} \left( { 5 \choose 3 } k \right) =  50k^2+15k+1 \geq 48k^2+13k+1 =  E_{P_{\mathbf{F};5}^L} \left( { 3 \choose 3 } k \right).\]
\end{example}

\section{The spine of polytopes}\label{sec:Spine}
Since it's hard to understand the polytopes exactly, we inscribe more accessible polytopes and semi-algebraic sets within them.
For a vector $\mathbf{F}$ of $m$ graphs, the \emph{spine} is the generalized moment curve
\[ \{ (p^{e_1}, p^{e_2}, \ldots, p^{e_m} ) \mid 0\leq p \leq 1 \}, \]
where $e_i$ is the number of edges in $F_i$. A graph $G$ from the Erd\H{o}s-R\'enyi random graph model $\mathcal{G}(n,p)$ have $n$ vertices, and edges are included independently with probability $p$. 
\begin{proposition}\label{prop:spine}
Let $\mathbf{F}=(F_1,\ldots,F_d)$ be a vector of graphs of order at most $n$, then the spine $\{ (p^{e_1}, p^{e_2}, \ldots, p^{e_d} ) \mid 0\leq p \leq 1 \}$  is in $P_{\mathbf{F};n}$.
\end{proposition}
\begin{proof}
The expected value of $t(F,G)$ with $G$ from $\mathcal{G}(n,p)$ is $p^e$ if $F$ have $e$ edges, since the edges of $G$ are included independently with probability $p$. For every instance of graphs $G$ from $\mathcal{G}(n,p)$ the vector $t(\mathbf{F},G)$ is a point in $P_{\mathbf{F};n}$. The expected value
of $t(\mathbf{F},G),$
\[ (p^{e_1}, p^{e_2}, \ldots, p^{e_m} ), \]
is then also a point in $P_{\mathbf{F};n}$.
\end{proof}

In Figure~\ref{fig:spine} is the polytope $P_{(K_3,C_4,K_4 \setminus e );6)}$ from Figure~\ref{fig:polytope} drawn with its spine.
\begin{figure}[htbp]
  \begin{center}
    \includegraphics[width=0.75\textwidth]{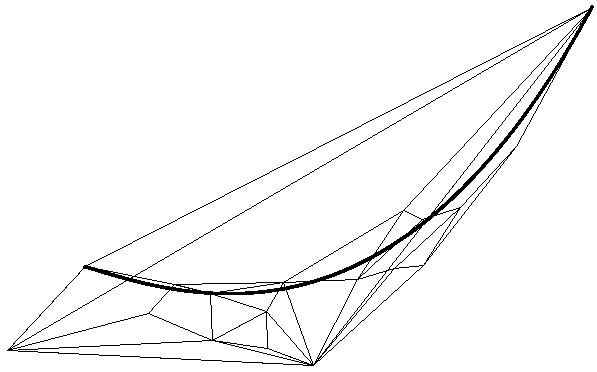}
    \caption{The polytope in Figure~\ref{fig:polytope} with its spine.}
    \label{fig:spine}
  \end{center}
\end{figure}
The point of Proposition~\ref{prop:spine} is that the spine is a generalized moment curve inside $P_{\mathbf{F};n}$. The convex hull of a finite number of points on the spine is a cyclic polytope. This can be seen directly by using generalized van der Monde matrices instead of the ordinary one in Ziegler's textbook derivation of the combinatorial structure of cyclic polytopes \cite{ziegler}. This shows that there are cyclic polytope inscribed in $P_{\mathbf{F};n}$. 

The convex hull of all of the spine is not a polytope, but its boundary can be algebraically described. In Figure~\ref{fig:cyclic} is the spine from Figure~\ref{fig:spine} drawn with its convex hull. Since the boundary structure of the convex hull in Figure~\ref{fig:cyclic}
\begin{figure}[htbp]
  \begin{center}
    \includegraphics[width=0.75\textwidth]{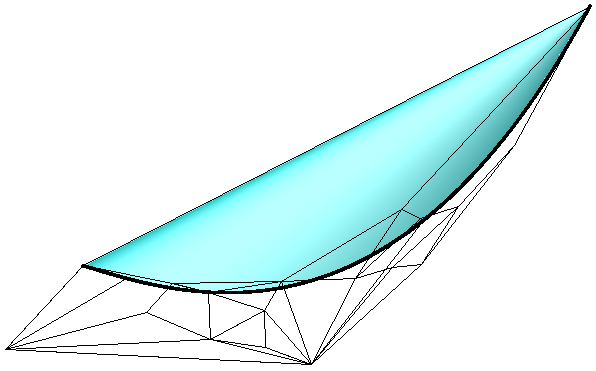}
    \caption{The spine in Figure~\ref{fig:spine} drawn with its convex hull.}
    \label{fig:cyclic}
  \end{center}
\end{figure}
is not very clear from this angle, we include in Figure~\ref{fig:coolCurve} the same spine with its convex hull, but from another perspective.
\begin{figure}[htbp]
  \begin{center}
    \includegraphics[width=0.75\textwidth]{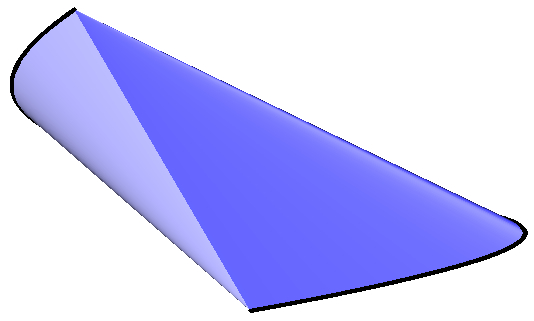}
    \caption{The spine with its convex hull from  Figure~\ref{fig:cyclic} drawn from another perspective.}
    \label{fig:coolCurve}
  \end{center}
\end{figure}

\subsection{Volumes}\label{sec:Volumes}
Inside our polytopes we have convex hulls of generalized moment curves and their volumes bound the volumes of polytopes from subgraph statistics. For the ordinary moment curve $\{ (p,p^2,\cdots, p^d) \mid 0 \leq p \leq 1 \}$ the volume of its convex hull was calculated by Karlin and Shapley \cite{karlinShapley}. 
To calculate the volume of the convex hull of the spine, and for later applications, we need the following standard approximation.
\begin{lemma}\label{lemma:volumeConvergingConvexHulls}
Let $A \subseteq B \subseteq [0,1]^d$ be sets so that for any $b \in B$ there is an $a \in A$ with
$| a-b | \leq \varepsilon$, then
\[ \textrm{Vol}(\textrm{conv}\, B) - \textrm{Vol}(\textrm{conv}\, A) \leq k_d\varepsilon \]
where $k_d$ only depends on the dimension $d$.
\end{lemma}
\begin{proof}
For $X=A,B$  and $\xi \in \mathbb{R}^d$ set 
\[ H_X(\xi)=\sup_{x \in X} \langle x,\xi \rangle = \sup_{x \in \textrm{conv}\,X} \langle x,\xi \rangle.  \]
By assumption, there is always a point within distance $\varepsilon$ in $A$ from each point in $B$,
so $H_B(\xi)-\varepsilon \leq H_A(\xi) \leq H_B(\xi)$ for all $\xi$. Using the setup of
H\"ormander \cite{hormander}, section 4.3, the function $H_C(\xi)=H_B(\xi)-\varepsilon$ defines a
convex body $C=\{x \mid \langle x, \xi \rangle \leq H_C(\xi), \xi \in  \mathbb{R}^d \}
\subseteq \textrm{conv}\, A \subseteq  \textrm{conv}\, B$ with
\[ \textrm{Vol}( \textrm{conv}\, B ) - \textrm{Vol}(C) \leq \varepsilon S_B \]
where $S_B$ is the surface area of the closed convex hull of $B$. The surface area can be estimated
from above with a $k_d$ since the volume of $B$ is at most one, for example as in Section I.8.3 of \cite{barvinok}. 
\end{proof}

\begin{proposition}\label{proposition:convergenceOfCyclic}
Let $e_1 < e_2 < \cdots < e_d,$ and $n$ be positive integers. Consider the cyclic polytope $P$
on the vertices
\[ \{ (x^{e_1},x^{e_2},\ldots,x^{e_d}) \mid x=i/n, \,\,i=0,1,2,\ldots,n \} \]
and the convex hull $C$ of the moment curve
\[ \{ (x^{e_1},x^{e_2},\ldots,x^{e_d}) \mid  0\leq x \leq 1\}. \]
Then
\[0\leq \textrm{Vol}(C) - \textrm{Vol}(P) \leq \frac{k_d}{n} (e_1+e_2+\cdots+e_d),\]
where $k_d$ is a constant depending only on the dimension $d$.
\end{proposition}
\begin{proof}
Let $f:[0,1]\rightarrow \mathbb{R}^d$ be $x\mapsto (x^{e_1},x^{e_2},\ldots,x^{e_d})$. From the vector valued
mean value theorem
\[ | f(y)-f(x) | \leq |y-x| \sup_{ 0 \leq t \leq 1} | f'(y+t(x-y)) | \]
applied with $y=\lfloor xn \rfloor /n$ and
$ | f'(y+t(y-x)) | \leq | f'(1) | = e_1+e_2+\cdots+ e_d$, we get that any
point on the moment curve is within distance
\[ \varepsilon = \frac{e_1+e_2+\cdots+ e_d}{n} \]
of a vertex of the polytope $P$. Now use Lemma~\ref{lemma:volumeConvergingConvexHulls}.
\end{proof}

We treat Schur polynomials as they are defined by fractions of generalized Vandermonde matrices instead of using representation theory. For these elementary facts we refer to Sagan \cite{sagan}. We will also use the Pfaffian, and refer to \cite{forresterBook,mehta} for definitions and basic identities.


\begin{theorem}\label{theorem:volumeConvexHullSpine}
Let $m>1$ and let $\mathbf{F}=(F_1,\ldots,F_{2m})$ be a vector of $2m$ distinct graphs with edges. The volume of the convex hull of the $2m$--dimensional spine of $\mathbf{F}$
\[\mathrm{Vol}(\mathrm{conv}\, \{  ( p^{e_1}, p^{e_2},  \ldots,p^{e_{2m}} ) \mid  0 \leq p \leq 1 \} )\]
is
\[\frac{1}{(2m)!m!}\int_{[0,1]^m} S_\lambda(x_1,x_1,x_2,x_2,\ldots,x_m,x_m)  \prod_{0\le i<j\le m}(x_i-x_j)^4 \,\, \mathrm{d}x_1\cdots \mathrm{d}x_m , \]
where $\lambda_i+2m-i = e_i$ in the Schur polynomial $S_\lambda$, and it is assumed that all $e_i$ are different and greater than $0$.

And in the odd dimensional case   $\mathbf{F}=(F_1,\ldots,F_{2m+1})$ we get 
\[\frac{1}{(2m+1)!m!}\int_{[0,1]^m} S_\lambda(x_1,x_1,x_2,x_2,\ldots,x_m,x_m,1)  \prod_{0\le i<j\le m}(x_i-x_j)^4  \prod_{0\le i\le m}(1-x_i)^2\,\, \mathrm{d}x_1\cdots \mathrm{d}x_m , \]
where $\lambda_i+2m+1-i= e_i $ in the Schur polynomial $S_\lambda$, and it is assumed that all $e_i$ are different and greater than $0$.
\end{theorem}
\begin{proof}
The idea of the proof is to approximate the volume of the convex hull of the spines with the volume of cyclic polytopes. The cyclic polytopes used to approximate the volume have their vertices placed uniformly distributed on the spine. When the number of vertices of these polytopes go to infinity the volume of the polytopes converge to the volume of the convex hull of the spine.

To compute the volume of a cyclic polytope $P$ we begin by triangulating the polytope $P$ in a special way. The volume of the polytope $P$ is then computed by summing the volumes of the simplexes in the triangulation. The triangulation used is constructed as follows: Pick a vertex $v$ in $P$, the triangulation is the one consisting of all the simplexes spanned by a facet of $P$ together with $v$. This is a triangulation of $P$ since cyclic polytopes are simplicial. Note that the facets containing $v$ do not contribute to the volume and so they can be ignored.

Consider the cyclic polytope $P_n$ spanned by $\{( (i/n)^{e_1}, (i/n)^{e_2},  \ldots,(i/n)^{e_{2m}} ) \mid  0\le i\le n\} $. The polytope $P_n$ has vertices uniformly distributed on the spine of $\mathbf{F}$. When $n$ goes to infinity the volume of $P_n$ converge to the volume of the convex hull of the spine. Choose the vertex $v$ to be $\mathbf{0}$. By Gale's evenness condition \cite{ziegler} the volume of $P_n$ is
\[
\frac{1}{(2m)!} \sum_{i_1,\ldots,i_{m}} \det \left[ \left( \frac{i_k}{n} \right)^{e_j}  \,\, \left( \frac{i_k+1}{n} \right)^{e_j}   \right]_{1 \leq j \leq 2m \atop 1 \leq k \leq m}
\]
where the sequences summed over are those satisfying $i_j+1<i_{j+1}$, $0<i_1$ and $i_{m}<n$. 

 The determinants can be expressed in terms of Schur polynomials and so the volume of the polytope $P_n$ is
\[\frac{1}{(2m)!}\sum_{i_1,\ldots,i_{2m}}S_\lambda\left(\frac{i_1}{n},\frac{i_1+1}{n},\ldots,\frac{i_m}{n},\frac{i_m+1}{n}\right)  \frac{1}{n^m}\prod_{1\le j<k\le m}\frac{(i_k-i_j)^2((i_k-i_j)^2-1)}{n^4},\]
where the sum is over the same sequences as in the formula with determinants. This is a Riemann sum and so when $n$ go to infinity this converges to the integral
\[\frac{1}{(2m)!}\int_I  S_\lambda(x_1,x_1,\ldots,x_m,x_m)  \prod_{1\le i<j\le m}(x_i-x_j)^4\,\, \mathrm{d}x_1\cdots \mathrm{d}x_m,\]
where $I=\{  0<x_1<x_2<\cdots<x_m<1 \}$. Both the Schur polynomial and the product of four-powers of differences are
symmetric polynomials, and we can replace $I$ by $[0,1]^m$ and divide by a factor $m!$ to get the integral stated in the
theorem.

The integral is the volume of the convex hull of the spine according to Proposition~\ref{proposition:convergenceOfCyclic}.

The odd dimensional case is similar, but every full dimensional simplex contains the vertex $\mathbf{1}$ since the facets of a cyclic polytope 
in odd dimensions all contain $\mathbf{0}$ or $\mathbf{1}.$ 
\end{proof}

In the even dimensional case there is a nice formula for the integral in Theorem~\ref{theorem:volumeConvexHullSpine}. 
\begin{theorem}\label{thm:For}
Let $\lambda=(\lambda_1,\ldots,\lambda_{2m})$ in the Schur polynomial $S_\lambda$, then
\[ \int_{[0,1]^m} S_\lambda(x_1,x_1,x_2,x_2,\ldots,x_m,x_m)  \prod_{0\le i<j\le m}(x_i-x_j)^4 \,\, \mathrm{d}\mathbf{x}=m!\prod_{1\le i<j\le 2m}\frac{y_i-y_j}{y_i+y_j}\]
where $y_i=2m+\lambda_i-i.$

In the odd dimensional case, let $\lambda=(\lambda_1,\ldots,\lambda_{2m+1})$ in the Schur polynomial $S_\lambda$. Then
\[\int_{[0,1]^m} S_\lambda(x_1,x_1,x_2,x_2,\ldots,x_m,x_m,1)  \prod_{0\le i<j\le m}(x_i-x_j)^4 \prod_{0\le i\le m}(1-x_i)^2 \,\, \mathrm{d}\mathbf{x}=m!\prod_{1\le i<j\le 2m+1}\frac{y_i-y_j}{y_i+y_j}, \]
where $y_i=2m+1+\lambda_i-i.$

\end{theorem}
\begin{proof}
First we prove it in the even dimensional case.
The first step of the proof is to go back to expressing the integrand as a limit of a determinant. This limit is also a nice determinant that is possible to integrate.
Observe that $S_\lambda(x_1,x_1,x_2,x_2,\ldots,x_m,x_m)  \prod_{1\le i<j\le m}(x_j-x_i)^4$ is the limit of 
\[\frac{\det[x_i^{2m-j+\lambda_j}]_{1 \leq i,j \leq 2m}}{\det[x_i^{2m-j}]_{1 \leq i,j \leq 2m}}\frac{\det[x_i^{2m-j}]_{1 \leq i,j \leq 2m}}{\prod_{i=1}^m(x_{2i}-x_{2i-1})}
 = \frac{\det[x_i^{2m-j+\lambda_j}]_{1 \leq i,j \leq 2m}}{\prod_{j=1}^m(x_{2i}-x_{2i-1})}
\]
as $x_{2i-1}\rightarrow x_{2i}$.
Row operations on the matrix $[x_i^{2m-j+\lambda_j}]_{i,j=1,\ldots,2m}$ don't change the determinant, and subtraction of row $2i-1$ from $2i$ for
$i=1,2,\ldots, m$ yields
\[\frac{\det\left[\begin{matrix}   x_{2i}^{2m-j+\lambda_j}-x_{2i-1}^{2m-j+\lambda_j}  \\  x_{2i-1}^{2m-j+\lambda_j}  \end{matrix}
\right]_{1\leq i\leq m \atop 1\leq j\leq 2m}}{\prod_{j=1}^m(x_{2i}-x_{2i-1})}
=\det\left[\begin{matrix}  (x_{2i}^{2m-j+\lambda_j}-x_{2i-1}^{2m-j+\lambda_j})/(x_{2i}-x_{2i-1})  \\  x_{2i-1}^{2m-1-j+\lambda_j}  \end{matrix}\right]_{1\le i\le m\atop 1\le j\le 2m}.\]
Using L'H\^opital's rule as $x_{2i-1}\rightarrow x_{2i}$ for $i=1,\ldots m$, and then relabeling $x_{2i}\rightarrow x_i$, we get that the integrand in the theorem statement equals
\[\det\left[\begin{matrix}(2m-j+\lambda_j)x_i^{2m-1-j+\lambda_j} \\x_i^{2m-j+\lambda_j}\end{matrix}\right]_{1\le i\le m\atop 1\le j\le 2m}.\]
This determinant is a sum over the symmetric group $S_{2m}$ 
\[\sum_{\pi \in S_{2m}}\varepsilon(\pi)\prod_{i=1}^m(2m-\pi(2i)+\lambda_{\pi(2i)})x_i^{2m-1-\pi(2i)+\lambda_{\pi(2j)}}x_i^{2m-\pi(2i-1)+\lambda_{\pi(2i-1)}},\]
where $\varepsilon(\pi)$ is the sign of $\pi$. If we require that $\pi(2j)>\pi(2j-1)$ for all $j$ then the integrand becomes
\[
\begin{array}{rl}
\displaystyle \sum_{\pi \in S_{2m}, \,\, \pi(2j)>\pi(2j-1)}\varepsilon(\pi) \,\,\,\,\prod_{i=1}^m & \Big(
((2m-\pi(2i)+\lambda_{\pi(2i)})-(2m-\pi(2i-1)+\lambda_{\pi(2i-1)})) \\
& \quad  x_i^{2m-1-\pi(2i)+\lambda_{\pi(2i)}}x_i^{2m-\pi(2i-1)+\lambda_{\pi(2i-1)}} \Big)
\end{array}
\]
and with the notation $y_i=2m+\lambda_i-i$ it equals
\[ 
 \sum_{\pi \in S_{2m}, \,\, \pi(2j)>\pi(2j-1)}\varepsilon(\pi) \prod_{i=1}^m 
 ( y_{\pi(2i)} - y_{\pi(2i-1)} ) x^{ y_{\pi(2i)} + y_{\pi(2i-1)} -1} 
 \]
 Integrating over each $x_j$ separately gives the following value of the integral in the theorem:
 \[ 
 \sum_{\pi \in S_{2m}, \,\, \pi(2j)>\pi(2j-1)}\varepsilon(\pi) 
 \frac{  y_{\pi(2i)} - y_{\pi(2i-1)} } {  y_{\pi(2i)} + y_{\pi(2i-1)} }.
 \]
The sum is actually the Pfaffian of an anti symmetric matrix, in general for $A=[a_{i,j}]_{i,j=1,\ldots,2m}$ anti symmetric
\[\mathrm{Pf}[A]= \frac{1}{m!} \sum_{\pi \in S_{2m},\,\, \pi(2j)>\pi(2j-1)}\varepsilon(\pi)\prod_{i=1}^ma_{\pi(2i-1),\pi(2i)}.\]

Hence
\[ \int_{[0,1]^m} S_\lambda(x_1,x_1,x_2,x_2,\ldots,x_m,x_m)  \prod_{0\le i<j\le m}(x_i-x_j)^4 \,\, \mathrm{d}\mathbf{x}=
m! \mathrm{Pf}\left[\frac{y_i-y_j}{y_i+y_j}\right]_{1 \leq i,j \leq 2m},\]
but we know that in general this Pfaffian evaluates to the stated product formula, a result usually attributed to Schur~\cite{schurpfaffian}.

The first odd case is true. For higher odd, start off as in the even but then expand the matrix along its column with only ones. This gives a sum of instances from the smaller even dimensional case, and by a degree argument the formula follows.
\end{proof}

\begin{corollary}
Let $\mathbf{F}$ be a vector of $k$ distinct graphs with edge counts $e_1 \geq \cdots \geq e_{k}>0$. Then the volume of the convex hull of the $k$--dimensional spine of $\mathbf{F}$ is
\[\mathrm{Vol}(\mathrm{conv}\, \{  ( p^{e_1}, p^{e_2},  \ldots,p^{e_{k}} ) \mid  0 \leq p \leq 1 \} ) =  \frac{1}{k!} \prod_{1\le i<j\le k}\frac{e_i-e_j}{e_i+e_j}.\]
\end{corollary}

Special cases of this integral have been computed before. Karlin and Shapley~\cite{karlinShapley} used the Selberg integral formula~\cite{forresterXXX,selberg} to do the case of $e_i=i$, but they arrived to the integral in a completely different way. Selberg's formula can only be applied in the cases where the edges are of consecutive magnitudes: $e_i=e+i.$

\subsection{Duality}\label{sec:Dual}
As for polytopes there is a duality theory for convex hulls of algebraic sets \cite{barvinok}. The dual of the convex hull of the moment curve $\{ (p,p^2, \ldots, p^n) \mid 0 \leq p \leq 1 \}$ parametrizes the degree $n$ polynomials that are non-negative on the interval $[0,1]$. The convex hulls of generalized moment curves are inside polytopes from subgraphs statistics, so the polytopes can be used to certify that polynomials are non-negative.
\begin{proposition}
Let $P$ be a polytope containing the generalized moment curve $\{ (p^{e_1}, p^{e_2}, \ldots, p^{e_d}) \mid 0 \leq p \leq 1 \}.$
If $ \langle(c_1,c_2, \ldots, c_d) ,v\rangle \geq -1 $
for all vertices $v$ of $P$ then the polynomial $q(x)=1+c_1x^{e_1}+c_2x^{e_2} + \cdots + c_d x^{e_d}$ is non-negative on the interval $[0,1]$.
\end{proposition}
\begin{proof}
The value of the polynomial $q(x)=1+c_1x^{e_1}+c_2x^{e_2} + \cdots + c_d x^{e_d}$ at the point $x=p$ is $1+\langle(p^{e_1}, p^{e_2}, \ldots, p^{e_d}), (c_1,c_2, \ldots, c_d)\rangle$. This implies that if $\langle(p^{e_1}, p^{e_2}, \ldots, p^{e_d}), (c_1,c_2, \ldots, c_d)\rangle\geq -1$ for all $p$ in $[0,1]$ then the polynomial $q$ is positive on $[0,1]$.
\end{proof}

\begin{example}

Our running example $P_{ (K_3, C_4, K_4 \setminus e ); 6}$ in Figure~\ref{fig:polytope} is perhaps not the most interesting polytope to certify non-negativity with, but we use it in an example anyways. 
The polynomial $p(x)=1-\frac{16}{3}x^3+\frac{11}{2}x^4-\frac{1}{2}x^5$ is non-negative on $[0,1]$ since $( -\frac{16}{3},\frac{11}{2},-\frac{1}{2})\cdot v \geq -1$ for all vertices $v$ of  $P_{ (K_3, C_4, K_4 \setminus e ); 6}$. Note that the point $(-\frac{16}{3},\frac{11}{2},-\frac{1}{2})$ is dual to the facet with vertices $(8/20,10/45,16/90),$ $(10/20, 15/45, 30/90),$ $(5/20, 3/45, 6/90)$, which one can find using Figure~\ref{fig:skeleton}. 

\end{example}

\section{Curvy zonotopes}\label{sec:Zonotope}
In this section we generalize the ideas used to construct spines. A very general class of random graph models called exchangeable random graph models is used instead of the Erd\H{o}s-R\'enyi  model, but the idea is the same: The expectation values from random graph models give easily parameterized sets inside the polytopes of subgraph statistics, and these sets are useful to derive information about the polytopes. The exchangeable random graph models are related to the graph limits developed by Lovasz and Szegedy \cite{lovaszSzegedy}, as explained by Diaconis and Janson \cite{diaconisJanson}. The machinery of graph limits is used more heavily in the next section where more of it is explained. In this section only the random graph models are needed.

The exchangeable random graph models is a generalization of the Erd\H{o}s-R\'enyi random graph model $\mathcal{G}(n,p),$ obtained by replacing $p$ by a measurable and symmetric function $W:[0,1]^2\rightarrow [0,1]$. Denote the set of these measurable symmetric function by $\mathcal{W}$.

\begin{definition}
Let $W$ be a function in $\mathcal{W}$. A graph $G$ from the \emph{exchangeable random graph model} $\mathcal{G}(n,W)$ is given as follows. The vertex set of $G$ is $[n]$.  Let $X_1,\ldots,X_n$ be independent random variables with uniform distribution on $[0,1]$. For every pair $1\leq i<j\leq n$ add an edge $ij$ to $G$ with probability $W(X_i,X_j).$
\end{definition}
Note that for given $X_1,\ldots,X_n$ the probability to get a graph with edge set $E$ from $\mathcal{G}(n,W)$ is
\[\prod_{ij\in E}W(X_i,X_j)\prod_{ij\in E(K_n)\setminus E}(1-W(X_i,X_j)).\]

We collect some basic facts from \cite{lovaszSzegedy}. To proceed further the expectation values of subgraph densities is computed for the exchangeable graph models. 

\begin{proposition}\label{prop:expect}
Let $F$ be a graph with vertex set $[m]$ and let $W\in\mathcal{W}$. The expected value of $t(F,G)$ when $G$ is from $\mathcal{G}(n,W)$ is
\[\mathbb{E}(t(F,G))=\int_{[0,1]^{m}}{\prod_{ij\in E(F)}W(x_i,x_j)}\mathrm{d}x_1\cdots \mathrm{d}x_m,\]
where it is assumed that $m\le n$
\end{proposition}

We use the notation $t(F,W)=\int_{[0,1]^{m}}{\prod_{ij\in E(F)}W(x_i,x_j)}\mathrm{d}x_1\cdots \mathrm{d}x_m$ for the expectation values calculated in Proposition~\ref{prop:expect}. For a vector of graphs $\mathbf{F}=(F_1,\ldots,F_d)$ we define $t(\mathbf{F},W)=(t(F_1,W),\ldots,t(F_d,W))$. 

In general it might be hard to compute $t(\mathbf{F},W)$ explicitly.

\begin{definition}
A $W \in \mathcal{W}$ is a \emph{stepfunction}, if for some partition of $[0,1]$ into intervals, the value of $W(x,y)$ is completely
determined by which parts of $[0,1]$ that $x$ and $y$ belongs to.

For any symmetric $n \times n$ matrix $M$ with entries in $[0,1]$, the stepfunction $W_M$ takes the value of $M$ at row
$i$ and column $j$ at $[(i-1)/n,i/n) \times [(j-1)/n,j/n)$.
\end{definition}
Note that the value of $W_M$ is not defined on the entire boundary of $[0,1]^2$, this do not matter for our purposes since the boundary have measure zero.

The stepfunctions are useful since any function in $ \mathcal{W}$ can be approximated by a stepfunction, and for stepfunctions $W$ it is possible to establish a polynomial expression for $t(\mathbf{F},W)$.

\begin{proposition}\label{prop:pol}
Let $M=[m_{i,j}]_{n\times n}$ be a symmetric matrix with entries in $[0,1]$ and let $\mathbf{F}$ be a vector of graphs. Then
\[t(\mathbf{F},W_M)=(p_{F_1;n}(M), p_{F_2;n}(M), \ldots, p_{F_d;n}(M)),\]
where
\[ p_{F;n}(M)
= \frac{1}{n^{|F|}} \sum_{\phi : V(F) \rightarrow [n]} \,\, \prod_{ij \in E(F)} m_{\phi(i),\phi(j)}.
\]

\end{proposition}
\begin{proof}
Let $F$ be a graph with vertex set $[k]$. Recall that the expectation value $t(F,W_M)$ is
\[\int_{[0,1]^{k}}{\prod_{ij\in E(F)}W_M(x_i,x_j)}\mathrm{d}x_1\cdots \mathrm{d}x_k.\]
Partition $[0,1]^{k}$ into the parts $[(i_1-1)/n,i_1/n) \times [(i_2-1)/n,i_2/n)\times\cdots \times [(i_k-1)/n,i_k/n)$, the boundary have measure zero and can be ignored. Observe that there is a bijection between the functions $\phi : V(F) \rightarrow [n]$ and the parts in the partition of $[0,1]^{k}$. The function $\phi$ corresponds to the part
$[(\phi(1)-1)/n,\phi(1)/n) \times [(\phi(2)-1)/n,\phi(2)/n)\times\cdots  \times [(\phi(k)-1)/n,\phi(k)/n)$. 

The function $\prod_{ij\in E(F)}W_M(x_i,x_j)$ is constant $\prod_{ij\in E(F)}m_{\phi(i),\phi(j)}$ on the part corresponding to $\phi$. The integral of  $\prod_{ij\in E(F)}W_M(x_i,x_j)$ over the part corresponding to $\phi$ is then $\frac{1}{n^{|F|}}\prod_{ij \in E(F)} m_{\phi(i),\phi(j)}$.

The integral over $[0,1]^{k}$ then splits into the sum
\[\int_{[0,1]^{k}}{\prod_{ij\in E(F)}W_M(x_i,x_j)}\mathrm{d}x_1\cdots \mathrm{d}x_k=\frac{1}{n^{|F|}} \sum_{\phi : V(F) \rightarrow [n]} \,\, \prod_{ij \in E(F)} m_{\phi(i),\phi(j)}.\]
\end{proof}
The polynomials $p_{F;n}(M)$ in Proposition~\ref{prop:pol} can be computed and evaluated, this makes it possible to study their image as a subset of the polytopes.

\begin{definition}\label{def:cz}
Let $\mathbf{F}=(F_1,\ldots,F_d)$ be a vector of graphs and $n$ a positive integer. The \emph{curvy zonotope} is
\[ Z_{\mathbf{F};n} = \left\{ (p_{F_1;n}(\mathbf{x}), p_{F_2;n}(\mathbf{x}), \ldots, p_{F_d;n}(\mathbf{x}))  
\mid \mathbf{x} \in [0,1]^{n^2},x_{ij}=x_{ji} \right\} \]
where
\[ p_{F;n}(\mathbf{x})
= \frac{1}{n^{|F|}} \sum_{\phi : V(F) \rightarrow [n]} \,\, \prod_{ij \in E(F)} x_{\phi(i)\phi(j)}.
\]
The curvy zonotope $Z_{(K_3,C_4,K_4\setminus e);2}$ is drawn in Figure~\ref{fig:toblerone}.
\end{definition}
\begin{figure}[htbp]
  \begin{center}
    \includegraphics[width=0.75\textwidth]{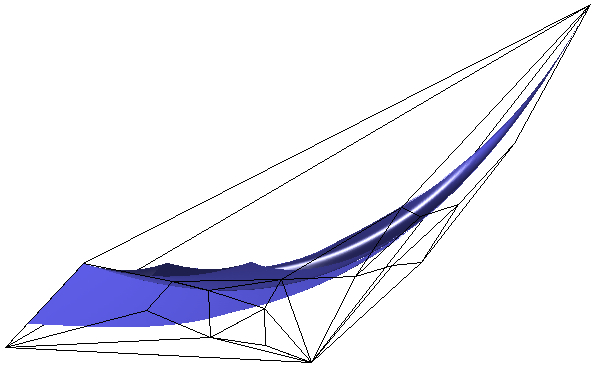}
    \caption{The curvy zonotope $Z_{(K_3,C_4,K_4\setminus e);2}$. As defined in Definition~\ref{def:cz}.}
    \label{fig:toblerone}
  \end{center}
\end{figure}

The curvy zonotope $Z_{\mathbf{F};n}$ consists of the expected values from random graphs models according to Proposition~\ref{prop:pol}, and this implies that it is a subset of the polytopes $P_{\mathbf{F};m}$. The curvy zonotopes contain the spines as the spines come from the special case of the function $W$ being constant $p$. The curvy zonotopes have many desirable properties that  the spine do not, for example full dimensionality.
\begin{theorem}\label{theorem:fullDim}
Let $\mathbf{F}=(F_1,\ldots,F_d)$ be a vector of $d$ distinct graphs with no isolated vertices and with $n_0\ge\max\{|F_1|,\ldots,|F_d|\}$. If $n\ge n_0$ then the curvy zonotope $Z_{\mathbf{F};n}$ is not contained in a hyperplane in $\mathbb{R}^d$.
\end{theorem}
\begin{proof}
The curvy zonotope is defined in terms of polynomials coming from graphs. The idea of the proof is to identify a term in the polynomials with a special exponent. This special exponent ensure that no linear combination of the polynomials can be constant and this proves that the curvy zonotope is not contained in a hyperplane.

Let $F$ be a graph in $\mathbf{F}$. The graph $F$ give a polynomial \[p_{F;n}(\mathbf{x})\frac{1}{n^{|F|}} \sum_{\phi : V(F) \rightarrow [n]} \,\, \prod_{ij \in E(F)} x_{\phi(i)\phi(j)}\] as in the definition of the curvy zonotope $Z_{\mathbf{F};n}$.

Assume that the curvy zonotope  is in the hyperplane defined by $C_{|F_1|}p_{F_1;n}(\mathbf{x})+\cdots+C_{|F_d|}p_{F_d;n}(\mathbf{x})=C$. Let $F$ be a graph with maximal order among the graphs $F_i$ with $C_{F_i}\neq 0$.

Let $\phi:V(F)\rightarrow [n]$ be injective.  The function $\phi$ give a term $\frac{1}{n^{|F|}}\prod_{ij\in E(F)}x_{\phi(i)\phi(j)}$ in the polynomial $p_{F;n}(\mathbf{x})$. That $\phi$ is injective implies that the graph $F'$, with the indices occurring in $\prod_{ij\in E(F)}x_{\phi(i)\phi(j)}$ as vertex set and the pairs of indices occurring in $\prod_{ij\in E(F)}x_{\phi(i)\phi(j)}$ as edge set, is isomorphic to $F$.

Let $F'$ be a graph with $C_{F'}\neq 0$ and let $p_{F';n}$ contain $\frac{1}{n^{|F|}}\prod_{ij\in E(F)}x_{\phi(i)\phi(j)}$. The function $\phi$ was injective from $V(F)$ and so $|F'|$ must be at least as large as $|F|$, but $|F|$ was maximal and hence $|F|=|F'|$. The function $\phi':V(F')\rightarrow [n]$ giving the term $\frac{1}{n^{|F|}}\prod_{ij\in E(F)}x_{\phi(i)\phi(j)}$ is then injective. But it is possible to reconstruct $F$ from $\frac{1}{n^{|F|}}\prod_{ij\in E(F)}x_{\phi(i)\phi(j)}$, assuming that the function giving the term is injective. This implies that no other polynomial from a graph $F'$ with $C_{F'}$ contains the term $\prod_{ij\in E(F)}x_{\phi(i)\phi(j)}$. This ensure that no term can cancel the unique term $C_F\frac{1}{n^{|F|}}\prod_{ij\in E(F)}x_{\phi(i)\phi(j)}$ in the polynomial $C_{|F_1|}p_{F_1;n}(\mathbf{x})+\cdots+C_{|F_d|}p_{F_d;n}(\mathbf{x})$, and so the polynomial can not be constant. This is a contradiction and the curvy zonotope is not contained in a hyperplane.

\end{proof}

The curvy zonotope is not in a hyperplane, this means that there are points in it that span $d$--dimensional simplex and then all the polytopes $P_{\mathbf{F};n}$ contain this simplex and is full dimensional. This establishes the following corollary:

\begin{corollary}
Let $\mathbf{F}=(F_1,\ldots,F_d)$ be a vector of $d$ distinct graphs with no isolated vertices and with $n_0\ge\max\{|F_1|,\ldots,|F_d|\}$.  There is a constant $V_\mathbf{F}>0$ such that that the volume of $P_{\mathbf{F};n}$ is larger than $V_\mathbf{F},$
for $n\ge n_0$.
\end{corollary}

The volume of the polytopes do decrease with increasing $n$. But it is not the case that the volume can become arbitrarily small. 

\begin{figure}[htbp]
  \begin{center}
    \includegraphics[width=0.75\textwidth]{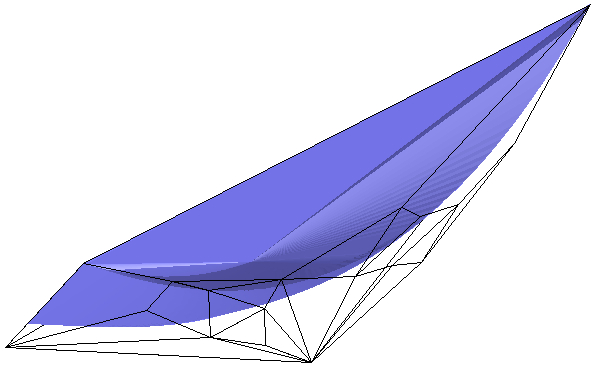}
    \caption{}
    \label{fig:convexToblerone}
  \end{center}
\end{figure}

\section{Limit object}\label{sec:Limit}
Recall that the polytopes of subgraph statistics satisfy the inclusion $P_{\mathbf{F},n}\subseteq P_{\mathbf{F},n_0}$ when $n_\le n$. This makes it possible to define a limit object.
\begin{definition}\label{definition:limit}
Let $\mathbf{F}$ be a vector of graphs where no entry have more than $n_0$ vertices. Define the \emph{limit object}
\[P_{\mathbf{F};\infty}=\bigcap_{n\ge n_0}P_{\mathbf{F};n}.\]
\end{definition}
Describing the polytopes is very hard for most $n$, but for very large $n$ the polytopes should essentially look like the limit object. Finding good descriptions of the limit objects and their properties is the goal of this section. As it turns out, the object are closely related to graph limits and the exchangeable random graph models. 

It seems almost hopeless to get a  simple description of the  polytopes $P_{\mathbf{F};n}$ for any nontrivial vector $\mathbf{F}$ and integer $n$. For the limit objects the situation is much better. Known results in extremal graph theory can be used to get good descriptions of the limit objects.

Razborov found in \cite{razborov1} the exact minimal possible triangle density for given edge density, when the number of vertices in the graph go to infinity. Razborov used machinery called flag algebras that he developed in \cite{razborov2} which is related to graph limits. If $T(p)$ is the minimal possible triangle density for edge density $p$, then $\{(p,T(p))\mid p\in[0,1]\}$ is a curve in the limit object $P_{(K_2,K_3),\infty}$. The convex hull of the curve $\{(p,T(p))\mid p\in[0,1]\}$ is the limit object $P_{(K_2,K_3),\infty}$. The convex hull of the curve has a nice explicit description: it's the hull of $\{(p,T(p))\mid p\in[0,1]\}$ is conv $\{(0,0),(1,1),(1-\frac{1}{2},0),(1-\frac{1}{3},\frac{2}{9}),\ldots,(1-\frac{1}{k},\frac{(k-1)(k-2)}{k^2}),\ldots\}$.

In this section the theory of graph limits is used to reinterpret the limit object, and to give a way to study it when an explicit description is not known.

In \cite{lovaszSzegedy} limits of certain sequences of graphs were defined. A sequence of graphs $(G_i)_{i=1,2,\ldots}$ is said to converge if the sequence $(t(F,G_i))_{i=1,2,\ldots}$ converges for every graph $F$. One of the important facts about convergent sequences of graphs is that they look like they come from exchangeable random graph models. One result in this direction from \cite{lovaszSzegedy} is the following:
\begin{theorem}[Lovasz-Szegedy \cite{lovaszSzegedy}]\label{thm:equiv}
If the sequence of graphs $(G_i)_{i=1,2,\ldots}$ converges, then there is a $W \in \mathcal{W}$ such that $t(F,W)=\lim_{i\rightarrow \infty}t(F,G_i)$ for all graphs $F$.
Let $W \in \mathcal{W}$ and let $(G_i)_{i=1,2,\ldots}$ be a sequence of random graphs generated by the exchangeable random graph model $\mathcal{G}(i,W)$, then $(G_i)_{i=1,2,\ldots}$ converges almost surely. In fact $t(F,W)=\lim_{i\rightarrow \infty}t(F,G_i)$ almost surely for all graphs $F$. 
\end{theorem}

The main ingredient in the proof is a method to pick a subsequence of $(G_i)_{i=1,2,\ldots}$ with Szemer\'edi regularity properties. Using this subsequence it is possible to construct a suitable $W$.

There is a useful norm on the space of functions $[0,1]^2\rightarrow \mathbb{R}$.
\begin{definition}
The \emph{rectangle norm} of a function $W:[0,1]^2\rightarrow \mathbb{R}$ is defined by
\[||W||_\square=\sup_{A\subseteq [0,1]^2}\int_AW(x_1,x_2)\mathrm{d}x_1\mathrm{d}x_2.\]
\end{definition}

In the previous section we calculated with stepfunctions. In more general situations it might be hard to get exact results, but it is possible to approximate any measurable by a stepfunction.
\begin{lemma}\label{lemma:approxWithStepfunction}
For any vector $\mathbf{F}$ of $d$ graphs with at most $e$ edges, and any $W \in \mathcal{W}$, there is a stepfunction $U$ with at most $2^{k_d \varepsilon^{-2}}$ steps (all of the same width) so that
\[ | t(\mathbf{F},W)-t(\mathbf{F},U) | < de \varepsilon. \]
Here $k_d$ is a fixed constant depending on $d$.
\end{lemma}
\begin{proof}
This is essentially a straightforward application of the weak Szemer\'edi regularity lemma contained in \cite{friezeKannan}, as explained in \cite{lovaszSzegedy}.
There is a constant $k$, so that for any $\varepsilon>0$  and $W \in \mathcal{W}$, there exists a stepfunction $U$ (with all steps of the same width) and at most $2^{k \varepsilon^{-2}}$ steps, so that $||U-W||_{\square} < \varepsilon.$ The rectangle norm used, relates to the fact, that, according to Lemma 4.1 of \cite{lovaszSzegedy},
$|t(F,U)-t(F,W)|\leq |E(F)| ||U-W||_{\square}$.
\end{proof}

One part of the proof of Theorem~\ref{thm:equiv} uses the following very useful result.
\begin{lemma}[Lovasz-Szegedy \cite{lovaszSzegedy}]\label{lemma:convergentGraphs}
Let $(G_i)_{i \geq 1}$ be a sequence of graphs. Then there is an infinite subsequence $(G_i)_{i\geq 1, i\in S}$ and a 
$W \in \mathcal{W}$ so that 
\[ \lim_{i \rightarrow \infty, i\in S} t(F,G_i) \]
exists for all graphs $F$, and it equals $t(F,W)$.
\end{lemma}

The limit object, the intersection of polytopes of subgraph statistics, is related to the graph limits of Lovasz and Szegedy. The limit object is the convex hull of expectation values.
\begin{theorem}\label{theorem:allW}
For every vector $\mathbf{F}$ of graphs, 
$P_{\mathbf{F};\infty} = \mathrm{conv}\,  \{   t(\mathbf{F},W)  \mid W \in \mathcal{W} \}.$
\end{theorem}
\begin{proof}
Denote with $P_{\mathbf{F};\mathcal{W}}$ the convex hull defined in the theorem statement.

First we prove that any $x\in P_{\mathbf{F};\infty}$ is in $P_{\mathbf{F};\mathcal{W}}$. By definition
$P_{\mathbf{F};\infty}$ is the intersection of the $d$-dimensional polytopes $P_{\mathbf{F};n}$
for $n$ is larger than some $n_0$. By triangulating the interior of every $P_{\mathbf{F};n}$
using only boundary vertices, we can describe $x$ in a simplex for every $n > n_0:$
\[ x= \sum_{j=0}^d \alpha_{j,n}v_{j,n} \]
where each $v_{j,n}$ is a vertex of $P_{\mathbf{F};n}$, each $\alpha_{j,n}\geq 0$, and $\sum_{j=0}^d \alpha_{j,n} = 1.$
For every vertex in a $P_{\mathbf{F};n}$ there is a graph realizing that subgraph statistics.
Choose  $d+1$ graph sequences $(G^j_n)^{0\leq i \leq d}_{n > n_0}$ so that 
$t(\mathbf{F},G^j_n)=v_{j,n}$. 

Use Lemma~\ref{lemma:convergentGraphs} on $(G^0_n)_{n>n_0}$ to find an index set $S_0$ and a symmetric
measurable function $W_0 : [0,1]^2 \rightarrow [0,1]$ so that 
$ \lim_{n \rightarrow \infty, n\in S_0} t(F,G_n^0) $
exists for all graphs $F$, and it equals $t(F,W_0)$.

For $j=1,2,\ldots, d$ repeat the use Lemma~\ref{lemma:convergentGraphs}, but on $(G^j_n)_{n\in S_{j-1}}$ to find an index set $S_j$ and a symmetric measurable function $W_j : [0,1]^2 \rightarrow [0,1]$ so that 
$ \lim_{n \rightarrow \infty, n\in S_j} t(F,G_n^j) $
exists for all graphs $F$, and it equals $t(F,W_j)$.

The points $t(\mathbf{F},W_0), \ldots, t(\mathbf{F},W_d)$ spans (a possibly degenerate) simplex in $P_{\mathbf{F};\mathcal{W}}$
and we should locate $x$ in it. The point $(\alpha_{0,n}, \alpha_{1,n}, \ldots, \alpha_{d,n})$ is in a $d$-simplex for all $n$, and
hence there is a subsequence $S'\subset S_d$ so that $\lim_{n \rightarrow \infty, n\in S'} \alpha_{j,n} = \tilde\alpha_j$.

Collecting the preceding limits, gives the desired
\[ x= \sum_{j=0}^d \tilde\alpha_j  t(\mathbf{F},W_j). \]

Now we prove the other direction. Pick a point $x$ in $P_{\mathbf{F};\mathcal{W}}$ parametrized by
\[ x= \sum_{j=0}^d \alpha_j  t(\mathbf{F},W_j) \]
and assume that $x\not\in P_{\mathbf{F};\infty}$. Since $P_{\mathbf{F};\infty}$ is a closed convex set
there is a hyperplane $H$ strictly separating $x$ and $P_{\mathbf{F};\infty}$.
By definition $P_{\mathbf{F};\infty}$ is the intersection of  $P_{\mathbf{F};n_0} \supset P_{\mathbf{F};n_0+1}
\supset  P_{\mathbf{F};n_0+2} \supset \cdots$ for some $n_0$.

Assume that all polytopes $P_{\mathbf{F};n}$ intersect the hyperplane $H$, and pick a convergent sequence $\{x_i\}_{i=n_0,n_0+1,\ldots}$ such that $x_i\in P_{\mathbf{F};i}\cap H$. The sequence $\{x_i\}_{i=n_0,n_0+1,\ldots}$ converge to something in $H$ and the  polytopes are compact, this implies that the limit of  $\{x_i\}_{i=n_0,n_0+1,\ldots}$ is in all polytopes $P_{\mathbf{F};i}$. This is a contradiction, $H$ will not intersect all $P_{\mathbf{F};n}$. For $n_1$ large enough, the hyperplane $H$ will not intersect $P_{\mathbf{F};n_1}$, and $x$ and $P_{\mathbf{F};n_1}$ are separated by $H$.

Generate random graphs $G_{n}^j \in \mathcal{G}(n,W_j)$ for $n\geq n_1$. According to Lemma~\ref{lemma:convergentGraphs},
$\lim_{n\rightarrow \infty} t(\mathbf{F},G^j_n)$ exists and converges to $t(\mathbf{F},W_j)$. Note that
all of the points $t(\mathbf{F},G^j_n)$ are in the closed convex set $P_{\mathbf{F};n_1}$, and hence
so is all their limits $t(\mathbf{F},W_j)$. But then $x= \sum_{j=0}^d \alpha_j  t(\mathbf{F},W_j)$
should also be in $P_{\mathbf{F};n_1}$, which contradicts that $x$ and $P_{\mathbf{F};n_1}$
are separated by a hyperplane.
\end{proof}

The curvy zonotopes can also be used to approximate the limit object arbitrary well.
\begin{theorem}
Let $\mathbf{F}$ be a vector of $d$ graphs on at most $e$ edges. If $\varepsilon > 0$ and $n>  2^{                 k_d   \frac{d^2e^2}{\varepsilon^2}}$ then
\[  0 \leq  \mathrm{Vol}( P_{\mathbf{F} ; \infty})  -  \mathrm{Vol}(  \mathrm{conv}\,  Z_{\mathbf{F};n}  ) < \varepsilon. \] Here $k_d$ is a constant depending only on $d$.
\end{theorem}
\begin{proof}
Any point in the limit object $P_{\mathbf{F},\infty}$ is in the convex hull of
\[ \{   t(\mathbf{F},W)  \mid W:[0,1]^2 \rightarrow [0,1] \textrm{ symmetric and measurable} \},  \]
 by Theorem~\ref{theorem:allW}, and the points of $Z_{\mathbf{F},n}$ is a subset of those,
since they  are all on the form $t(\mathbf{F},U) $ for some stepfunction $U$ with $n$ steps.
According to Lemma~\ref{lemma:approxWithStepfunction} any $t(\mathbf{F},W)$ is within distance
$k_d^{-1}\varepsilon$ of a point in the curvy zonotope $Z_{\mathbf{F},n}$, and then according
to Lemma~\ref{lemma:volumeConvergingConvexHulls} we are done.
\end{proof}

\section{Conjectures about the limit object $P_{\mathbf{F};\infty}$.}\label{sec:Conjectures}

We have previously inscribed cyclic polytopes in the limit objects. We will now define another cyclic polytope and conjecture that a particular class of limit objects actually are cyclic polytopes.

The vertices of $P_{(K_2,K_n);\infty}$ are given by the limits of complete $k$-equipartite graphs according to results by Bollobas \cite{Bol1,Bol2}. It is not hard to see that $P_{(K_2,K_n);\infty}$ is a cyclic polytope, and we believe that this is true in a more general setting.

For positive integers $e_1< e_2 < \ldots < e_m$ define the \emph{tail} $s^{\mathbf{e}} :[0,1] \rightarrow [0,1]^m $ by
\[ s^{\mathbf{e}}_i(x) = \prod_{j=1}^{e_i-1} (1-jx). \]
\begin{proposition}
The convex hull of any finite set of points on a tail is a cyclic polytope.
\end{proposition}
\begin{proof}
Let $e_i=i+1$. 
Consider the matrix whose columns are the points on the polytope, ordered by increasing $x$. The polynomial on row $i$ is of degree $i$ and has nonzero coefficients on all terms of lower degree. There is then a sequence of row operations that takes this matrix to a matrix with the same determinant but the polynomial on row $i$ is a monomial of degree $i$. In this situation the argument from Zieglers book \cite{ziegler} applies and the polytope is cyclic.

The general case, with arbitrary $e_i$ is obtained from the case $e_i=i+1$ by projection: any tail is obtained from the tail with $e_i=i+1$  by removing some coordinates. The tail with $e_i=i+1$ is generic and this is the property which ensure that points on it span cyclic polytopes, removing coordinates from a generic curve give a new generic curve.
\end{proof}

\begin{conjecture} \label{conj:massaK}
Let $e_1< e_2 < \ldots < e_m$ be positive integers and $s^{\mathbf{e}}$ their tail. The convex hull of $\mathbf{1}$ and
$ \{ s^{\mathbf{e}}(1/k) \mid k=1,2,3,\ldots \} $ is $P_{(K_{e_1},K_{e_2},\ldots,K_{e_m}); \infty }$.
\end{conjecture}
The conjectured vertex description of $P_{(K_{e_1},K_{e_2},\ldots,K_{e_m}); \infty }$ also gives a facet description since it's essentially a cyclic polytope.
If we would chop off the vertex $\mathbf{1}$ from the convex hull described in Conjecture~\ref{conj:massaK} with an hyperplane, then the remaining convex set would be an ordinary polytope. We believe this is true in the following general form.
\begin{conjecture}
For any vector $\mathbf{F}$ of graphs there is a positive integer $m$, such that for any $\varepsilon >0$, the limit object $P_{\mathbf{F};\infty}$ can be chopped down to a polytope with a finite number of vertices, by using $m$ hyperplanes to remove at most a volume $\varepsilon$.
\end{conjecture}

\section*{Acknowledgement}\label{sec:acknowledgement}

In a manuscript of this article we calculated several special cases of an integral and conjectured that there was an easy combinatorial description. This was proved by Peter Forrester and we are very happy that he allows us to include his proof in Theorem \ref{thm:For}.

Alexander Engstr\"om is a Miller Research Fellow, and gratefully acknowledges support from the Miller Institute for Basic Research in Science at UC Berkeley. 

Patrik Nor\'en gratefully acknowledges support from the Wallenberg foundation.

\end{document}